\documentclass[10pt,a4paper,leqno]{amsart}
\usepackage{amssymb, amsmath, amscd}
\usepackage[dvipsnames]{xcolor}
\usepackage[hypertexnames=false]{hyperref}
\usepackage[normalem]{ulem}
\usepackage{color}
\usepackage{setspace} 

\DeclareMathOperator{\Hes}{\rm Hes}
\DeclareMathOperator{\Ric}{\rm Ric}

\newcommand{\nf}{\nabla f}

\newcommand{\KN}{\mathbin{\bigcirc\mspace{-15mu}\wedge\mspace{3mu}}}
\makeatletter
\@namedef{subjclassname@2020}{%
	\textup{2020} Mathematics Subject Classification}
\makeatother

\newtheorem{theorem}{Theorem}[section]
\newtheorem{lemma}[theorem]{Lemma}

\newtheorem{corollary}[theorem]{Corollary}
\newtheorem{example}[theorem]{Example}

\theoremstyle{definition}

\theoremstyle{remark}
\newtheorem{remark}[theorem]{Remark}

\newcommand\restr[2]{{
  \left.\kern-\nulldelimiterspace 
  #1 
  \vphantom{\big|} 
  \right|_{#2} 
  }}

\begin{document}
\title[Conformally weighted Einstein manifolds]{ Conformally weighted Einstein manifolds: the uniqueness problem}
\author{M. Brozos-V\'azquez, E. Garc\'ia-R\'io, D. Moj\'on-\'Alvarez}
\address{MBV: CITMAga, 15782 Santiago de Compostela, Spain}
\address{\phantom{MBV:}  
	Universidade da Coru\~na, Campus Industrial de Ferrol, Department of Mathematics, 15403 Ferrol,  Spain}
\email{miguel.brozos.vazquez@udc.gal}
\address{EGR: CITMAga, 15782 Santiago de Compostela, Spain}
\address{\phantom{EGR:}
	Department of Mathematics, Universidade de Santiago de Compostela,
	15782 Santiago de Compostela, Spain}
\email{eduardo.garcia.rio@usc.es}
\address{DMA: CITMAga, 15782 Santiago de Compostela, Spain}
\address{\phantom{DMA:}
	Department of Mathematics, Universidade de Santiago de Compostela,
15782 Santiago de Compostela, Spain}
\email{diego.mojon.alvarez@usc.es}
\thanks{Research partially supported by grants PID2022-138988NB-I00 funded by MICIU/AEI/10.13039/501100011033 and by ERDF, EU, and ED431C 2023/31 (Xunta de Galicia, Spain); and by contract FPU21/01519 (Ministry of Universities, Spain).}
\date{}

\maketitle

\begin{abstract} 
We discuss smooth metric measure spaces admitting two weighted Einstein representatives of the same weighted conformal class. First, we describe the local geometries of such manifolds in terms of certain Einstein and quasi-Einstein warped products.
Secondly, a global classification result is obtained when one of the underlying metrics is complete, showing that either it is a weighted space form, a special Einstein warped product, or a specific family of quasi-Einstein warped products. As a consequence, it must be a weighted sphere in the compact case.
\end{abstract}

\section{Introduction}

Let $\mathcal{M}=(M^n,g,f,m,\mu)$ be a {\it smooth metric measure space} ({\it SMMS} for short), i.e., an $n$-dimensional Riemannian manifold $(M,g)$ endowed with a weighted volume element $e^{-f} dvol_g$ defined by a density function $f\in C^\infty(M)$, a dimensional parameter $m\in \mathbb{R}^+$ and an auxiliary curvature parameter $\mu\in \mathbb{R}$. 

Geometric objects defined on $\mathcal{M}$ are motivated by the geometry on the base of the formal warped product
\[
M\times_v F^m (\mu)=(M\times F^m, g\oplus v^2  h^F),
\]
where $v=e^{-f/m}$ is the warping function and $F^m(\mu)=(F^m, h^F)$ is an $m$-dimensional fiber of constant sectional curvature $\mu$. The volume element on $\mathcal{M}$ is $e^{-f} dvol_g$, thus  corresponding to the restriction of the volume to the base.

The study of the geometry of SMMSs relies on the use of weighted tensors such as the {\it $m$-Bakry-\'Emery Ricci tensor}, defined as
\begin{equation}\label{eq:Bakry-Emery-Ricci-tensor}
	\rho^m_f=\rho+\operatorname{Hes}_f-\frac{1}{m} df\otimes df,
\end{equation}
and the {\it weighted scalar curvature}, given by
\begin{equation}\label{eq:Weighted-scalar-curvature}
	\tau^m_f=\tau+2\Delta f-\frac{m+1}{m}||\nf||^2+m(m-1)\mu \,e^{\frac{2}m f},
\end{equation}
where $\Delta f$ is the Laplacian of $f$. The Bakry-Émery Ricci tensor is the Ricci tensor of $M\times_v F^m (\mu)$ restricted to the base,
whereas the weighted scalar curvature is the scalar curvature of $M\times_v F^m (\mu)$, taken as a smooth function on $M$. Notice that  the parameter $m=1$ makes the value of $\mu$ irrelevant, so in this case we refer to the SMMS by the quadruple $(M,g,f,1)$. This distinguished value will appear in the subsequent analysis. 

The weighted scalar curvature plays a role in the definition of the {\it weighted Schouten tensor} 
\begin{equation}\label{eq:Weighted-Schouten-tensor}
	P^m_f=\frac{1}{n+m-2}(\rho^m_f-J^m_f g),
\end{equation}
where the {\it weighted Schouten scalar} is given by $J^m_f=\frac{1}{2(n+m-1)}\tau^m_f$. 

Motivated by their variational behavior, Case (see \cite{Case-JDG,Case-Sigmak}) introduced {\it weighted Einstein} SMMSs as those which satisfy
\begin{equation}\label{eq:weighted-Einstein-equation}
P^m_f=\lambda g,
\end{equation}
for some constant $\lambda\in\mathbb{R}$. For constant density functions, one recovers the notion of standard Einstein manifolds so, whenever the density function is constant, the SMMS is said to be {\it trivial}. 
Moreover, when the relation \eqref{eq:weighted-Einstein-equation} holds, there exists a unique  $\kappa\in \mathbb{R}$ such that $J_f^m=(m+n)\lambda -m\kappa e^{\frac{f}m}$  (see  \cite{Case-JDG,Case-Sigmak}  and Lemma~\ref{lemma:scale}). This constant is called {\it scale}. Noticeably,  for non-trivial SMMSs, $\kappa=0$ if and only if $J^m_f$ is constant (and hence so is the weighted scalar curvature), in which case the manifold is \emph{quasi-Einstein}, i.e., $\rho^m_f=\beta g$ for some $\beta\in\mathbb{R}$ (we refer to Section~\ref{sect:preliminaries} and references \cite{Case-JDG,Case-Sigmak} for details).

Two SMMSs $\mathcal{M}=(M^n,g,f,m,\mu)$ and $\widehat{\mathcal{M}}=(\widehat M^n,\widehat g,\widehat f,m,\mu)$ are {\it isometric} if there exists a Riemannian isometry $\psi:(M,g)\to (\widehat M,\widehat g)$ preserving the densities, i.e., $f=\widehat f\circ \psi$. Thus, we will identify isometric SMMSs without further mention if there is no possible confusion. As such, we say that two SMMSs  are {\it  conformally equivalent} if there exists a smooth function $\phi \in C^\infty(M)$ such that $\widehat g=e^{-2\phi/m}g$ and $\widehat f=f+ \phi$ (see  \cite{Case-Sigmak}).

The existence of weighted Einstein manifolds in a given conformal class of SMMSs arises as a natural question. Indeed, for two conformally equivalent SMMSs $\mathcal{M}$ and $\widehat{\mathcal{M}}$, taking $u=e^{\phi/m}$ and using the transformation formulae in \cite{Case-Sigmak}, one has
\begin{equation}\label{eq:conf-trans-law-P}
	\widehat{P}_{\widehat f}^m=P_f^m+u^{-1}\Hes_u-\frac{1}2 u^{-2}||\nabla u||^2g.
\end{equation}
Thus, the problem of finding conformal classes which admit weighted Einstein representatives entails finding a Riemannian metric $g$, a density function $f$ and a conformal function $u$ defined on $M$ so that \eqref{eq:conf-trans-law-P} is satisfied for $\widehat{P}_{\widehat f}^m=\widehat\lambda u^{-2} g$ and some constant $\widehat \lambda$. This translates into finding a solution to the PDE 
\begin{equation}\label{eq:conf-trans-law-P2}
P_f^m+u^{-1}\Hes_u=	u^{-2}\left(\widehat\lambda  +\frac{1}2 ||\nabla u||^2\right)g
\end{equation}
for a constant $\widehat \lambda$. 

This turns out to be an unmanageable problem in general. However, it is reasonable to address the problem of non-uniqueness, this is, to find weighted Einstein SMMSs that admit another weighted SMMS in their conformal class (excluding homotheties, i.e. rescalings of the metric by a constant factor). These problem is relevant, for example, in the study of weighted Yamabe-type problems (see \cite{Case-JDG,Case-Sigmak}).
First results in this regard were given in \cite{Case-JDG} for $\mu=0$. On the one hand, it was shown that if $M$ is compact then the SMMS is conformal to the standard sphere $(S^n,g,0,1,0)$. On the other hand, if the manifold is complete and the weighted Schouten tensor vanishes, then it is isometric to Euclidean space with a particular family of possible density functions. Additionally, some partial results were given in \cite{Case-Sigmak} for specific families of SMMSs. In this note, we generalize these results and give a complete answer to this problem, determining all conformal classes which admit more than one weighted Einstein representative.  

The problem we address in this note is somewhat motivated by the question of finding two Einstein metric in the same conformal class. This is a classical problem whose first contributions date back to Brinkmann \cite{Brinkmann}, and literature on the topic is extensive. See, for example, \cite{Yano-Nagano, Kuhnel}, where the latter includes a detailed review of the subject. We will see that, in fact, some of the solutions we obtain are built on Einstein manifolds which conformally transform into another Einstein manifold. Thus, we recover part of the known results from the context where the density is constant, although in our case there is, additionally, a change of density. We emphasize, however, that there are also weighted Einstein solutions whose underlying manifold is not Einstein (see Theorem~\ref{th:conf-class-2WE-complete} below).

We begin by discussing the local geometric structure of manifolds which admit two weighted Einstein structures in the same weighted conformal class. It is convenient to consider the density as the positive function $v=e^{-\frac{f}{m}}$, so that a conformal change transforms the metric and the density as $\widehat{g}=u^{-2}g$ and $\widehat{v}=u^{-1}v$, respectively. It is worth noting that, since conformal transformations can transform non-trivial manifolds into trivial ones and vice versa, it is useful to include the trivial case in our analysis.

%

\begin{theorem}\label{th:conf-class-2WE-local}
	Let $(M^n,g,f,m,\mu)$ be a weighted Einstein SMMS, with $P_f^m=\lambda g$, and such that there exists a conformally equivalent weighted Einstein SMMS $(M^n,\widehat{g},\widehat{f},m,\mu)$, with $\widehat{P}_f^m=\widehat{\lambda} \widehat{g}$. Then, on a neighborhood of each regular point of the conformal factor $u$, $M$ decomposes as a warped product $I\times_\varphi N$, where $I\subset \mathbb{R}$ is an open interval and $\nabla u$ is tangent to $I$. Furthermore, one of the following holds:
	\begin{enumerate}
		\item  $(M,g)$ and $(M,\widehat{g})$ are Einstein with $\rho=2(n-1)\lambda g$ and $\widehat{\rho}=2(n-1)\widehat{\lambda} g$, and the density takes the form $f=-m\log(\varphi v_N+\alpha)$, where $v_N$ is a function on $N$ and $\alpha$ is a function on $I$. 
		
		Moreover, the fiber  $(N,g^N)$ is Einstein and there exist constants $\xi,\nu$ determined by $v$ and $u$ such that $\Hes^N_{v_N}=(\xi-(\nu^2-4\lambda \widehat\lambda) v_N)g^N$.
		
		\item $(M^n,g,f,m)$ and $(M^n,\widehat{g},\widehat{f},m)$ are quasi-Einstein with $\rho_f^m=2(m+n-1)\lambda g$ and $\widehat{\rho}_f^m=2(m+n-1)\widehat{\lambda} \widehat{g}$, and the density $f$ splits as $f= -m\log(\varphi)+f_N$ where  $f_N$ is a function on $N$. 
		
		Moreover, the fiber $(N,g^N, f_N,m)$ is quasi-Einstein too.
	\end{enumerate}
\end{theorem}	

Notice that, as an {\it a posteriori} consequence of Theorem~\ref{th:conf-class-2WE-local}, some of the manifolds of interest to this work have constant scalar curvature. Indeed, they are Einstein (case (1)). However, this does not hold in general for manifolds in case (2), which do have constant weighted scalar curvature (since they are quasi-Einstein), but not constant scalar curvature. In both cases, the corresponding constancy is preserved by the conformal transformation. 

The warped product structure of the SMMSs in Theorem~\ref{th:conf-class-2WE-local} reduces the study of equation~\eqref{eq:conf-trans-law-P2} to that of a system of ODEs involving the conformal factor $u$ and the warping function $\varphi$. Further study of the weighted Einstein equation for these structures provides information on the density function $f$ and the curvature parameter $\mu$ (see Section~\ref{sec:2confWE-local}). Moreover, we will see in Remark~\ref{re:split-v-eins-2WE} that the density function itself satisfies an Obata equation in the Einstein case, thus giving rise to an alternative splitting of the manifold. 

Secondly, we focus on global rigidity results for SMMSs with complete underlying manifold $(M,g)$. In this context, the admissible geometries with two weighted Einstein representatives of the same weighted conformal class are either weighted space forms (see definitions in Section~\ref{sect:preliminaries}) or special families of warped products. 

\begin{theorem}\label{th:conf-class-2WE-complete}
	Let $(M^n,g,f,m,\mu)$ be a complete SMMS such that $P_f^m=\lambda g$, with scale $\kappa$, and such that there exists a conformally equivalent weighted Einstein SMMS. Then, $(M,g,f,m,\mu)$ is isometric to one of the following SMMSs:
	\begin{enumerate}
		\item A weighted space form as described in Examples~\ref{ex:sphere}, \ref{ex:euclidean-space} and \ref{ex:hyperbolic-space}.
		\item  A warped product $\mathbb{R}\times_\varphi N$, with $N$ complete, and such that $\varphi(t)=A e^{t\sqrt{-2\lambda}}$, where $t$ parameterizes $\mathbb{R}$ by arc length. Moreover, $\lambda<0$ and one of the following holds:
		\begin{enumerate}
			\item $(M,g)$ is Einstein and  $(N,g^N)$ is Ricci-flat. The density function has the form 
			\[
			 f=-m\log\left(\frac{\kappa}{2\lambda}+Be^{t\sqrt{-2\lambda}}\right),
			\]
		for some $B\geq0$ and $\kappa\leq 0$. Moreover, $m=1$ or $\mu=-\frac{\kappa ^2}{2 \lambda }\geq0$.
			\item $(M,g,f,m)$ is quasi-Einstein, $f$ splits as $f=-m \log{\varphi}+f_N$, and $(N,g^N,f_N,m)$ is  also quasi-Eins\-tein with $(\rho^m_{f_N})^N=0$.
			\end{enumerate}
	\end{enumerate}
\end{theorem}

\begin{remark}
Manifolds in Theorem~\ref{th:conf-class-2WE-complete}~(2)(a) are geometrically significant from several points of view. Under this condition, $(N,g^N)$ is a Ricci-flat complete manifold and $(M,g)$ is Einstein with $\rho=2(n-1)\lambda g$. Hence, this manifold falls into Theorem~\ref{th:conf-class-2WE-local}~(1) alongside the weighted space forms in Theorem~\ref{th:conf-class-2WE-complete}~(1). Furthermore,the complete manifolds in Theorem~\ref{th:conf-class-2WE-complete}~(1) and (2)(a) are precisely those that admit a non-homothetic conformal change into another Einstein manifold (see \cite[Theorem~27]{Kuhnel}). Consequently, for trivial SMMSs, i.e., those with constant density function, we recover this well-known result. Strikingly, the non-trivial SMMSs in these families correspond exactly with those which are weighted Einstein and have harmonic weighted Weyl tensor, as shown in \cite{Brozos-Mojon-JMPA} (see Remark~\ref{re:conf-2WE-local-classification} and Section~\ref{sect:2WEconf-complete} for details). 
\end{remark}

As a consequence of Theorem~\ref{th:conf-class-2WE-complete}, we obtain a rigidity result for compact SMMSs that generalizes the result given in \cite{Case-JDG} for $\mu=0$. Moreover, from the point of view that smooth metric measure spaces generalize manifolds with constant density, we can say that it also extends \cite[Corollary 23]{Kuhnel} to the weighted setting.

\begin{corollary}\label{cor:2WEconf-compact}
	Let $(M^n,g,f,m,\mu)$ be a non-trivial compact weighted Einstein SMMS. If there exists a non-constant conformal factor such that the transformed manifold is weighted Einstein, then $(M,g,f,m,\mu)$ is an $m$-weighted $n$-sphere (which is conformally equivalent to a standard sphere with vanishing density). 
\end{corollary}

The outline of the paper is as follows. In Section~\ref{sect:preliminaries}, we go over some preliminaries on the  geometry of weighted Einstein manifolds and weighted conformal classes for SMMSs. We also discuss the analyticity of the conformal factor. 

In Section~\ref{sec:2confWE-local}, we perform a local analysis of the geometric structure of weighted Einstein manifolds which admit another weighted Einstein structure in their conformal class. From the transformation formula \eqref{eq:conf-trans-law-P}, we prove the splitting of these SMMSs as warped product with one-dimensional base. Further study of the transformation formula and the resulting weighted Einstein equations allows us to prove Theorem~\ref{th:conf-class-2WE-local}. Additionally, we provide some remarks and examples of SMMSs that give more insight into the types of geometries that arise for the different items of the theorem.

Finally, in Section~\ref{sect:2WEconf-complete}, we prove Theorem~\ref{th:conf-class-2WE-complete}, determining the complete manifolds with several weighted Einstein structures in the same conformal class. We end with the proof of  Corollary~\ref{cor:2WEconf-compact} and the analysis of the compact case.

\section{Preliminaries: weighted Einstein condition and conformal classes}\label{sect:preliminaries}
Throughout this section, let $(M,g)$ be an $n$-dimensional Riemannian manifold, with $n\geq 3$, and $f\in C^\infty(M)$. Also, let $m\in \mathbb{R}^+$ be a dimensional parameter and $\mu\in \mathbb{R}$ an auxiliary curvature parameter, so that we can define the smooth metric measure space $\mathcal{M}=(M,g,f,m,\mu)$. 


\subsection{Weighted Einstein SMMSs}

This notion was coined by J. Case (see \cite{Case-JDG,Case-Sigmak}), who described a SMMS $\mathcal{M}=(M,g,f,m,\mu)$ as \emph{weighted Einstein} if $P^m_f=\lambda g$ for some $\lambda\in\mathbb{R}$. The motivation for this definition comes from the fact that weighted Einstein manifolds with $\mu=0$ are critical points of the total weighted scalar curvature functional, which arises in the study of the weighted Yamabe problem (see \cite{Case-JDG}). Moreover, Case proved several results describing weighted Einstein manifolds as critical points of functionals related to the weighted $\sigma_k$-curvature (we refer to \cite{Case-adv-2016,Case-Sigmak} for details on the weighted $\sigma_k$-curvature and the behavior of weighted Einstein SMMSs).

A crucial fact about weighted Einstein SMMSs is that the Schouten scalar is intimately related to the density function as follows.
\begin{lemma} \cite{Case-JDG,Case-Sigmak} \label{lemma:scale}
	Let  $(M^n,g,f,m,\mu)$ be a weighted Einstein SMMS, i.e. with $P^m_f=\lambda g$ for some $\lambda\in \mathbb{R}$. Then, there is a unique  $\kappa\in \mathbb{R}$ such that $J_f^m=(m+n)\lambda -m\kappa e^{\frac{f}m}$. 
\end{lemma}

The real number $\kappa$ in Lemma~\ref{lemma:scale} is called {\it scale}. Notice that if $\kappa=0$, then $J^m_f=(m+n)\lambda$ and, from \eqref{eq:Weighted-Schouten-tensor}, one has  $\rho^m_f=2(n+m-1)\lambda g$, so the SMMS is quasi-Einstein. Even more, a  non-trivial weighted Einstein SMMS is quasi-Einstein if and only if the scale $\kappa$ vanishes. 

It is worth noting that, on the one hand, if a formal warped product $M\times_v F^m (\mu)$ is Einstein, then the base is necessarily quasi-Einstein. On the other hand, if a manifold $(M,g)$ is quasi-Einstein, then there exists a warped product of the form $M\times_v F^m (\mu)$ which is Einstein \cite{Kim-Kim-Warped}. This illustrates that, in general, the weighted Einstein condition is not inherited from the Einstein character of a corresponding warped product, which instead leads to the quasi-Einstein notion.

\subsection{Conformal classes}
Two SMMSs, $(M^n,g,f,m,\mu)$ and $(M^n,\widehat{g},\widehat{f},m,\mu)$, are said to be {\it  conformally equivalent} if there exists a smooth function $\phi\in C^\infty(M)$ such that $\widehat{g}=e^{-2\phi/m}g$ and $\widehat{f}=f+ \phi$ (see \cite{Case-Sigmak}). The definition is motivated by the two warped products  $(M,g)\times_v F^m (\mu)$ and $(M,\widehat g)\times_{\widehat{v}} F^m (\mu)$  being conformally equivalent as Riemannian manifolds, where $v=e^{-\frac{f}m}$ and $\widehat{v}=e^{-\frac{\widehat{f}}m}$. Thus, we can rephrase the definition saying that two SMMSs are conformally equivalent if there exists a positive $u\in C^\infty(M)$ such that $\widehat{g}=u^{-2}g$ and $\widehat{v}=u^{-1}v$. Notice that the definition can also be stated in local terms if instead of a global function $\phi$ (equivalently, $u$), there is a function defined only on a neighborhood of each point.

  Notice that direct products of the form $F^n(-\mu)\times F^m(\mu)$ are locally conformally flat, and those of the form $F^n(c)\times \mathbb{R}$ are locally conformally flat for any $c$ \cite{yau}. Thus, a SMMS $(M^n,g,f,m,\mu)$ with $m\neq 1$ is {\it locally conformally flat} (in the weighted sense) if it is locally conformally equivalent to $(F^n,h(-\mu),0,m,\mu)$. If $m=1$, $(M^n,g,f,1)$ is {\it locally conformally flat} if it is locally conformally equivalent to $(F^n,h(c),0,1)$ for some value of the sectional curvature $c$.

\subsection{Regularity of solutions}
Suppose that for two locally conformally equivalent SMMSs with $\widehat{g}=u^{-2}g$ and $\widehat{v}=u^{-1}v$, we have $P_f^m=\lambda g$ and $\widehat{P}_f^m=\widehat{\lambda}\widehat{g}$. 
It was shown in \cite{Brozos-Mojon-JMPA} that both the metric and the density function of weighted Einstein SMMSs are real analytic in harmonic coordinates. As a consequence, we will show that the conformal factor $u$ relating both of them is also analytic.

%
%

\begin{lemma}\label{le:u-analytic}
If $u$ is a solution of \eqref{eq:conf-trans-law-P2} on a weighted Einstein SMMS $(M,g,f,m,\mu)$, then $u$ is (real) analytic in harmonic coordinates on $M$.
\end{lemma}
\begin{proof}
 Since $(M,g,f,m,\mu)$ is weighted Einstein, we have  $P_f^m=\lambda g$ and \eqref{eq:conf-trans-law-P2} becomes 
	\begin{equation}\label{eq:hessian-conformal-function}
		\Hes_u=u\left(\widehat{\lambda} u^{-2}-\lambda +\frac{1}2 u^{-2}||\nabla u||^2\right)g,
	\end{equation}
	or, taking the trace of this equation,
	\[
	\Delta u+\, \mathrm{l.o.t}=0,
	\]
	where l.o.t. stands for lower order terms. In harmonic coordinates, this geometric equation becomes the quasi-linear second-order PDE $g^{rs}\frac{\partial^2 u}{\partial x^r \partial x^s}+\, \mathrm{l.o.t}=0$, which is elliptic.
	Moreover, since weighted Einstein metrics are (real) analytic in harmonic coordinates (see \cite[Theorem 2.5]{Brozos-Mojon-JMPA}), the equation is of the form $F(u, \partial u,, \partial^2 u)=0$, with $F$ real analytic. It follows that the conformal factor $u$ is real analytic in harmonic coordinates (see, for example, \cite[J.41]{Besse}). 
\end{proof}

As a result of the previous lemma, the set of regular points of $u$ is open and dense in $M$.  We will use this fact in subsequent sections.

\subsection{Weighted space forms}

The weighted analogues of the usual space forms (complete weighted Einstein manifolds realized on the three model spaces) played a key role in the classification of weighted Einstein manifolds with weighted harmonic Weyl tensor in \cite{Brozos-Mojon-JMPA}, and will do so again in this work. We present them as follows.

\begin{example}[$m$-weighted $n$-sphere $(\mathbb{S}^n(2\lambda),g_{\mathbb{S}}^{2\lambda},f_m,m,\mu)$]\rm\label{ex:sphere} 
	Let $(\mathbb{S}^n(2\lambda),g_{\mathbb{S}}^{2\lambda})$ be the $n$-sphere of constant sectional curvature $2\lambda>0$ (equivalently, of radius $\frac{1}{\sqrt{2\lambda}}$), with the standard round metric
	\[
	g_{\mathbb{S}}^{2\lambda}=dt^2+(2\lambda)^{-1}\sin^2(t \sqrt{2\lambda}) g_{\mathbb{S}^{n-1}}, \quad t\in \left(0,\tfrac{\pi}{\sqrt{2\lambda}}\right),
	\]
	 where $t$ denotes the geodesic distance from the pole $N$ of the sphere and $\mathbb{S}^{n-1}$ is the $(n-1)$-sphere of radius 1. This metric extends smoothly to the poles $N$ and $-N$. Take the positive density $v(t)= A+B\cos(\sqrt{2\lambda} t)$ for $A\in \mathbb{R}^+$, $B\in \mathbb{R}$ such that $A>|B|$ and define, by continuity, $v(N)=A+B$ and $v(-N)=A-B$. For $m\neq 1$, fix $\mu=2\lambda(B^2-A^2)$. Then, the SMMSs $(\mathbb{S}^n(2\lambda),g_{\mathbb{S}}^{2\lambda},f_m,m,\mu)$ and $(\mathbb{S}^n(2\lambda),g_{\mathbb{S}}^{2\lambda},f_1,1)$, where $f_m=-m \log v$, are weighted Einstein with $P_f^m=\lambda g$ and scale $\kappa=2\lambda A>0$. Hence, these weighted spheres are only quasi-Einstein in the trivial case, where $\rho_f^m=\rho=2(n-1)\lambda$.
\end{example}

\begin{example}[$m$-weighted $n$-Euclidean space $(\mathbb{R}^n,g_\mathbb{E},f_m,m,\mu)$]\rm\label{ex:euclidean-space} 
	Let $(\mathbb{R}^n,g_\mathbb{E})$ be the standard Euclidean space, whose metric can be written as a warped product as
	\[
	g_\mathbb{E}=dt^2+t^2 g_{\mathbb{S}^{n-1}}, \quad t\in (0,\infty),
	\]
	extending smoothly to $t=0$. Consider the positive density $v(t)= A+Bt^2$ with $A\in \mathbb{R}^+$, $B\in[0,\infty)$. For $m\neq 1$, set the parameter $\mu=-4AB$. The SMMSs $(\mathbb{R}^n,g_\mathbb{E},f_m,m,\mu)$ and $(\mathbb{R}^n,g_\mathbb{E},f_1,1)$ are weighted Einstein with $P_f^m= 0$ and scale $\kappa=2 B\geq 0$, so they are quasi-Einstein only when they are trivial.
\end{example}

\begin{example}[$m$-weighted $n$-hyperbolic space $(\mathbb{H}^n(2\lambda),g_{\mathbb{H}}^{2\lambda},f_m,m,\mu)$]\rm\label{ex:hyperbolic-space} 
	We denote by $(\mathbb{H}^n(2\lambda),g_{\mathbb{H}}^{2\lambda})$ the $n$-hyperbolic space  of constant sectional curvature $2\lambda<0$, with the metric
	\[
	g_{\mathbb{H}}^{2\lambda}=dt^2+(-2\lambda)^{-1}\sinh^2(t  \sqrt{-2\lambda}) g_{\mathbb{S}^{n-1}}, \quad t\in (0,\infty),
	\]
	extending smoothly to $t=0$. Take the positive density $v(t)=A+B\cosh(\sqrt{-2\lambda} t)$, with  $B\in [0,\infty)$, $A\in \mathbb{R}$ such that $A>-B$. Moreover, for $m\neq 1$, fix $\mu=2\lambda (B^2-A^2)$. Then, the SMMSs $(\mathbb{H}^n(2\lambda),g_{\mathbb{H}}^{2\lambda},f_m,m,\mu)$ and $(\mathbb{H}^n(2\lambda),g_{\mathbb{H}}^{2\lambda},f_1,1)$ are weighted Einstein with $P_f^m=\lambda g$ and scale $\kappa=2\lambda A$. Note that the scale can have any sign, depending on the value of $A$ and, in contrast to the two previous models, taking $A=0$ results in a family of non-trivial quasi-Einstein manifolds. 
\end{example}

\section{Conformal weighted Einstein SMMs: local study}\label{sec:2confWE-local}

We adopt notation from previous sections and begin the analysis from a local point of view without further assumptions.
Firstly, we see that the conformal factor satisfies a generalized Obata equation (compare with the proof of  \cite[Proposition 9.5]{Case-JDG}). This is the same equation satisfied by a conformal factor transforming an Einstein metric into another one (cf. \cite{Brinkmann, Kuhnel}) and it provides some information on the structure of the underlying manifold, which decomposes as a warped product. 			
			
\begin{lemma}\label{le:weighted-Einstein-conformal-Obata}
Let $(M^n,g,v,m,\mu)$ be a SMMS such that $P_f^m=\lambda g$ and $u$ a non-constant solution of \eqref{eq:conf-trans-law-P2}.
Then, the function $u$ is a solution of the generalized Obata equation
				\begin{equation}\label{eq:generalized-Obata-equation-conformal}
					\Hes_{u}+\gamma(u)g=0,
				\end{equation}
				with $\gamma(u)=2\lambda u-\nu$, for some constant $\nu\in \mathbb{R}$. Moreover, around any regular point of $u$, $(M,g)$ is locally isometric to a warped product $I\times_\varphi N$, where $I\subset \mathbb{R}$ is an open interval, $\nabla u$ is tangent to $I$, and $\varphi(t)=\pm u'(t)$, where $t$ is a local coordinate parameterizing $I$ by arc length.
\end{lemma}

\begin{proof}
Since $P_f^m=\lambda g$, the Hessian of $u$ is given by \eqref{eq:hessian-conformal-function}. Thus, it follows that the level hypersurfaces of $u$ around regular points are totally umbilical. Consequently, in a neighborhood of each regular point, $(M,g)$ splits as a twisted product $I\times_{\tilde{\varphi}} N$, for some function $\tilde{\varphi} $ defined on $I\times N$ (see \cite{Ponge-Twisted}). 
Moreover, let $E_1=\nabla u /||\nabla u||$. By \eqref{eq:hessian-conformal-function}, we have $g(\nabla_XY,E_1)=-\frac{1}{||\nabla u||}\Hes_u(X,Y)=H g(X,Y)$ for $X,Y \in \nabla u^\perp$, where $H=-\frac{u}{||\nabla u||}\left(\widehat{\lambda} u^{-2}-\lambda +\frac{1}2 u^{-2}||\nabla u||^2\right)$. Since $X(H)=0$ for all $X\in \nabla u^\perp$, the mean curvature vector field $H E_1$ is parallel in the normal bundle $\operatorname{span}\{\nabla u\}$, so the leaves of the fiber are spherical. Hence, the twisted product reduces to a warped product $I\times_\varphi N$ for some function $\varphi$ defined on $I$ (see \cite{Hiepko}). 

Let $t$ be a coordinate parameterizing $I$ by arc length. Then, evaluating equation \eqref{eq:hessian-conformal-function} in $(\partial_t,\partial_t)$ yields
\begin{equation}\label{eq:conformal-factor-hessian}
u^{-2}\left(\lambda u^2-\widehat{\lambda}-\frac{1}2(u')^2 +uu''\right)=0.
\end{equation}
Now, note that we can write
\[
	((u')^2u^{-1}+2\lambda u+2\widehat{\lambda}u^{-1})'=u^{-2}(2u''u-(u')^2-2\widehat{\lambda}+2\lambda u^2)u'=0
\]
where the middle expression vanishes by \eqref{eq:conformal-factor-hessian}. Hence $(u')^2u^{-1}+2\lambda u+2\widehat{\lambda}u^{-1}=\nu$ for some constant $\nu \in \mathbb{R}$, which yields $(u')^2=-2\lambda u^2+2\nu u-2\widehat{\lambda}$. Substituting the value of $(u')^2$ into \eqref{eq:hessian-conformal-function} we obtain $\Hes_{u}+(2\lambda u-\nu)g=0$, which is the generalized Obata equation \eqref{eq:generalized-Obata-equation-conformal} around regular points of $u$. Since the critical points of $u$ are isolated by Lemma~\ref{le:u-analytic}, by smoothness, equation~\eqref{eq:generalized-Obata-equation-conformal} extends to $M$. 

Moreover, for any unitary vector field $X\in \partial_t^\perp$, we use the warped product decomposition (see \cite{Oneill}) to compute $\Hes_u(X,X)=\frac{u'\varphi'}{\varphi}$. Hence  \eqref{eq:hessian-conformal-function} and \eqref{eq:conformal-factor-hessian} yield
\[
	\frac{u'\varphi'}{u\varphi}=\frac{\widehat{\lambda}}{u^2}-\lambda +\frac{(u')^2}{2u^2 }=\frac{u''}{u},
\]
from where $u'\varphi'-u''\varphi=0$, and then $\varphi=K u'$ for some $K\in\mathbb{R}$ such that $\varphi>0$ in $I$ (recall that $u$ is non-constant). Rescaling the metric in $N$, we can assume $K=1$ if $u'>0$ and $K=-1$ if $u'<0$.
\end{proof}

%

Notice that, as pointed out in \cite[Proposition 4]{Kuhnel}, conformal changes given by solutions to equation~\eqref{eq:generalized-Obata-equation-conformal} preserve the constancy of the scalar curvature, which is a crucial fact in the study of conformal transformations preserving the Einstein character. However, in the context of this note, manifolds do not have constant scalar curvature in general.

Another essential difference between the usual setting and conformal changes that transform a weighted Einstein SMMS into another is the transformation of the density.  
As a next step we analyze the form of the density function assuming that the manifold decomposes as a warped product according to Lemma~\ref{le:weighted-Einstein-conformal-Obata}.

\begin{lemma}\label{le:conf-2WE-IorN-QE}
Let $(M^n,g,f,m,\mu)$, where $(M,g)=I\times_\varphi N$ and  $P_f^m=\lambda g$, admit a non-constant solution of \eqref{eq:conf-trans-law-P2}. Then $v=e^{-\frac{f}m}$ splits as
\[
v=\varphi(t) v_N(x_1,\dots,x_{n-1}) + \alpha(t)
\]
where $x_1,\dots x_{n-1}$ are coordinates of the fiber $N$.
\end{lemma}
\begin{proof}
By Lemma~\ref{le:weighted-Einstein-conformal-Obata}, we have that $u''+2\lambda u-\nu=0$ and $\varphi=\pm u'$. Hence $\frac{\varphi''}{\varphi}=-2\lambda$. Moreover, since $(M,g)=I\times_\varphi N$, the Ricci tensor takes the following form \cite{Oneill}:
\begin{equation}\label{eq:ricci-warped}
\begin{array}{rcl}
	\rho(\partial_t,\partial_t)&=&-(n-1)\frac{\varphi''}{\varphi}=2(n-1)\lambda, \qquad \rho(\partial_t,X)=0,  \\
	\noalign{\medskip}	
	\rho(X,Y)&=&\rho^N(X,Y)-\left(\frac{\varphi''}{\varphi}+(n-2)\frac{(\varphi')^2}{\varphi^2}\right)g(X,Y),
\end{array}
\end{equation}
for $X,Y\in \partial_t^\perp$ and $t$ parameterizing $I$ by arc length. Now, using the scale equation given by Lemma~\ref{lemma:scale}, the weighted Einstein equation~\eqref{eq:weighted-Einstein-equation} can be written as
\begin{equation}\label{eq:BE-Ricci-warped}
	\rho+\Hes_f-\frac1m df\otimes df=\left(2(m+n-1)\lambda-m\kappa e^{\frac{f}{m}}\right)g.
\end{equation}		
Under the change of variable $v=e^{-\frac{f}{m}}$, equation \eqref{eq:BE-Ricci-warped} takes the form
\begin{equation}\label{eq:BE-Ricci-warped2}
	\rho-mv^{-1}\Hes_v=(2(m+n-1)\lambda-m\kappa v^{-1})g.
\end{equation}	
Thus, taking $X$ the lift of a vector field in $N$ we have $-mv^{-1}\Hes_v(\partial_t,X)=0$, and computing this Hessian yields 
\[
0=\Hes_v(\partial_t,X)= \partial_tX(v)-\frac{\varphi'}{\varphi} X(v).
\]
 Thus,  locally either $X(v)=0$ or
\[
 \frac{\partial_t X(v)}{X(v)}=\frac{\varphi'}{\varphi} \text{ \; which implies } \partial_t \log(X(v))=(\log \varphi)',
\]
and integrating with respect to $t$ we have  
\[
\log X(v)= \log \varphi + \bar{v}^X_N(\overrightarrow{x})
\]
where  $\overrightarrow{x}=(x_1,\dots,x_{n-1})$ are coordinates on $N$, so $X(v)= \varphi (t) v^X_N(\overrightarrow{x})$ where $v^X_N$ does not depend on $t$, but may depend on the choice of $X$. Since $\varphi$ depends only on $t$ and there is no confusion, we skip it to simplify notation henceforth. Set notation $v_N^{\partial_{x_i}}=v_N^i$ and take $X=\partial_{x_1}$, so $\partial_{x_1}(v)=\varphi\, v^1_N(\overrightarrow{x}) $. Integrating this expression with respect to $x_1$ yields
\begin{equation}\label{eq:v1alpha1-2WEconf}
	 v= \varphi\, \tilde v^1_N(\overrightarrow{x}) + \alpha_1 (t,x_2,\dots, x_{n-1})
\end{equation}
for some function $\tilde v^1_N(\overrightarrow{x})$ on $N$ and some function $\alpha_1$ which does not depend on $x_1$. Now, we work with this expression to  see that the form in \eqref{eq:v1alpha1-2WEconf} can be rearranged so that $\alpha_1$ does not depend on $x_2$. We differentiate this expression with respect to $\partial_{x_2}$ to find
\[
\varphi\,v^2_N(\overrightarrow{x}) =\partial_{x_2}(v)=\varphi\, \partial_{x_2} \tilde v^1_N(\overrightarrow{x}) + \partial_{x_2}\alpha_1 (t,x_2,\dots, x_{n-1}).
\]
Hence,
\[
v^2_N(\overrightarrow{x})- \partial_{x_2} \tilde v^1_N(\overrightarrow{x}) =\varphi^{-1}\,\partial_{x_2}\alpha_1 (t,x_2,\dots, x_{n-1}).
\]
Since the left-hand side of this equation does not depend on $t$, differentiating with respect to $t$ gives 
\[
	0=\varphi^{-2}(\varphi\, \partial_t \partial_{x_2}\alpha_1 -\varphi' \partial_{x_2}\alpha_1),
\]
from where it follows that either $\alpha_1$ does not depend on $x_2$, and thus we attain our objective, or	$\frac{\partial_t \partial_{x_2}\alpha_1}{\partial_{x_2}\alpha_1 } =\frac{\varphi'}{\varphi}$. Assuming that the latter holds, we integrate with respect to $t$ to find
\[
\partial_{x_2}\alpha_1 (t,x_2,\dots, x_{n-1})=\varphi\, \gamma_1(x_2,\dots,x_{n-1})
\]
for some function $\gamma_1$ on $N$. Hence, we can write
\[
\begin{array}{rcl}
\alpha_1 (t,x_2,\dots, x_{n-1})&=&\varphi \int \gamma_1(x_2,\dots,x_{n-1})dx_2+ \alpha_2(t,x_3,\dots,x_{n-1})\\
\noalign{\medskip}
&=&\varphi\, \tilde\gamma_1(x_2,\dots,x_{n-1})+ \alpha_2(t,x_3,\dots,x_{n-1})
\end{array}
\]
for some $\tilde{\gamma}_1$ defined on $N$ and $\alpha_2$ not depending on $x_1$ and $x_2$. Substituting this value of $\alpha_1$ into \eqref{eq:v1alpha1-2WEconf}, we have
\[
\begin{array}{rcl}
 v&=& \varphi\, \tilde v^1_N(\overrightarrow{x}) + \alpha_1 (t,x_2,\dots, x_{n-1})\\
 \noalign{\medskip}
 &=&\varphi\, (\tilde v^1_N(\overrightarrow{x}) + \tilde\gamma_1(x_2,\dots,x_{n-1}))+ \alpha_2(t,x_3,\dots,x_{n-1})\\
 \noalign{\medskip}
 &=&\varphi\, \tilde v^2_N(\overrightarrow{x}) + \alpha_2(t,x_3,\dots,x_{n-1}), 
\end{array}
\]
where $\tilde v_N^2(\overrightarrow x) =\tilde v^1_N(\overrightarrow{x}) + \tilde\gamma_1(x_2,\dots,x_{n-1})$. Thus, we get that $v$ is of the form of \eqref{eq:v1alpha1-2WEconf} with $\alpha$ not depending on $x_1$ and $x_2$. Now, using $\partial_{x_i}(v)=v^i_N(\overrightarrow{x}) \varphi(t)$ for $i=3,\dots,n-1$, we repeat the process outlined above, eliminating the dependence on the $x_i$ variable from the corresponding $\alpha_i$ function. After a number of iterations, $v$ becomes
\[
v=\varphi\, v_N(x_1,\dots,x_{n-1}) + \alpha(t)
\]
for a function $v_N$ defined on the fiber and a function $\alpha$ defined on the base.
\end{proof}

For the form of the density function given by Lemma~\ref{le:conf-2WE-IorN-QE}, the remaining components of the weighted Einstein equation \eqref{eq:BE-Ricci-warped2} provide additional information on the density and the geometry of the underlying manifold, allowing us to prove the local classification result.

\bigskip


\noindent {\it Proof of Theorem~\ref{th:conf-class-2WE-local}}: Let $(M^n,g,f,m,\mu)$ be a SMMS such that $P_f^m=\lambda g$ and such that there exists a SMMS $(M^n,\widehat{g},\widehat{f},m,\mu)$, with $\widehat{g}=u^{-2}g$, $\widehat{v}=u^{-1}v$ (where $v=e^{-\frac fm}$ and $u$ is non-constant) and such that $\widehat{P}_f^m=\widehat{\lambda} \widehat{g}$. By Lemma~\ref{le:weighted-Einstein-conformal-Obata}, $(M,g)$ splits locally around the regular points of the conformal factor $u$ as a warped product $I\times_\varphi N$. Moreover, by Lemma~\ref{le:conf-2WE-IorN-QE}, the density takes the form $v=\varphi v_N+\alpha$, where $v_N$ is defined on $N$ and $\alpha$ is defined on $I$.

From  \eqref{eq:ricci-warped} we have $\rho(\partial_t,\partial_t)=2(n-1)\lambda$, so the weighted Einstein equation \eqref{eq:BE-Ricci-warped2} yields
\begin{equation}\label{eq:BE-Ricci-base}
	\rho_f^m(\partial_t,\partial_t)=2(n-1)\lambda-mv^{-1}\Hes_v(\partial_t,\partial_t)=2(m+n-1)\lambda-m\kappa v^{-1}.
\end{equation}
Hence $\Hes_v(\partial_t,\partial_t)=\partial_t^2v=-2\lambda v+\kappa$. Since $v=\varphi v_N+\alpha$, and knowing that $\varphi''=-2\lambda \varphi$, this implies $\alpha''=-2\lambda \alpha+\kappa$. Now, consider the following decomposition:
\[
	\Hes_v= v_N\Hes_{\varphi}+\varphi\Hes_{v_N}+dv_N\otimes d\varphi+d\varphi\otimes dv_N+ \Hes_{\alpha}. 
\]
Take two lifts $X,Y$ of vector fields in $N$. Since $\varphi$ and $\alpha$ are defined on $I$, we have
\[
	\Hes_v(X,Y)= v_N(\varphi')^2 \varphi g^N(X,Y)+\varphi\Hes_{v_N}(X,Y)+ \alpha' \varphi' \varphi g^N(X,Y). 
\]
Using this expression and \eqref{eq:ricci-warped}, equation \eqref{eq:BE-Ricci-warped2} for $X,Y$ reads
\begin{equation}\label{eq:conf-QE-fiber}
\begin{array}{rcl}
	\rho_f^m(X,Y)&=&\rho^N(X,Y)-\left(\varphi''\varphi+(n-2)(\varphi')^2\right) g^N(X,Y) \\
	\noalign{\medskip}	
	&&-mv^{-1}(v_N(\varphi')^2\varphi+\alpha'\varphi'\varphi ) g^N(X,Y) \\
	\noalign{\medskip}	
	&&-mv^{-1}\varphi\Hes_{v_N}(X,Y) \\
	\noalign{\medskip}
	&=&(2(m+n-1)\lambda-m\kappa v^{-1})\varphi^2g^N(X,Y).
\end{array}
\end{equation}
Let $X$ be a unit eigenvector of  the Ricci operator $\Ric^N$, and let $\rho^N(X,X)=g^N(\Ric^N (X),X)=r(X)$ be its associated eigenvalue. Also, denote $\Hes_{v_N}(X,X)=h(X)$ and notice that, by \eqref{eq:conf-QE-fiber}, $X$ is also an eigenvector of the Hessian operator. In order to simplify notation, we skip the $X$ dependence in the following calculations unless explicitly needed. From the formula
\[
\begin{array}{rcl}
h=\Hes_{v_N}(X,X)=g(\nabla_X\nabla v_N,X)&=&\varphi^2g^N(\nabla_X (\varphi^{-2}\nabla^Nv_N),X) \\
	\noalign{\medskip}
	&=&g^N(\nabla^N_X \nabla^Nv_N,X)=\Hes^N_{v_N}(X,X),
\end{array}
\]
it follows that $h$, as well as $r$, is also a function defined on the fiber $N$. Taking $Y=X$ in \eqref{eq:conf-QE-fiber} and multiplying by $v$, we obtain the following equation:
\begin{equation}\label{eq:conf-QE-fiber-hr}
\begin{array}{rcl}
	0&=&v r-v\left(\varphi''\varphi+(n-2)(\varphi')^2\right)-m(v_N(\varphi')^2\varphi+\alpha'\varphi'\varphi )   \\
	\noalign{\medskip}		
	&&-m\varphi h -(2(m+n-1)\lambda v-m\kappa )\varphi^2,
\end{array}
\end{equation}
where $v=\varphi v_N+\alpha$, $\varphi$ and $\alpha$ depend on $t$, and the remaining functions depend only on the coordinates $x_1,\dots, x_{n-1}$ of the fiber $N$. Now, differentiating \eqref{eq:conf-QE-fiber-hr} with respect to any $x_i$ and substituting $\varphi''=-2\lambda \varphi$ yields
\[
	\alpha \partial_{x_i}r+\varphi\left(\partial_{x_i}(v_N r)-m \partial_{x_i}h-(m+n-2)(2\lambda \varphi^2+(\varphi')^2) \partial_{x_i}v_N\right)=0.
\]
Note that the expression $2\lambda\varphi^2+(\varphi')^2$ is constant. Indeed, $(2\lambda \varphi^2+(\varphi')^2)'=2 \varphi' (2\lambda \varphi+\varphi'')=0$. Moreover, since $\varphi=\pm u'$, $u''=-2\lambda u+\nu$ and $(u')^2=-2\lambda u^2+2\nu u-2\widehat{\lambda}$ by Lemma~\ref{le:weighted-Einstein-conformal-Obata}, we can write
\begin{equation}\label{eq:const_beta_Eins_QE}
	2\lambda \varphi^2+(\varphi')^2=\nu^2 -4\lambda\widehat{\lambda}.
\end{equation}
Thus, the equation above becomes
\[
	\alpha \partial_{x_i}r+\varphi\left(\partial_{x_i}(v_N r)-m \partial_{x_i}h-(m+n-2)(\nu^2 -4\lambda\widehat{\lambda} )\partial_{x_i}v_N\right)=0.
\]
On a suitable open set, it follows that either $\partial_{x_i}r=0$ or 
\[
	\alpha\varphi^{-1}=-\frac{\partial_{x_i}(v_N r)-m \partial_{x_i}h-(m+n-2)(\nu^2 -4\lambda\widehat{\lambda} )\partial_{x_i}v_N}{\partial_{x_i}r}=A
\]
for some constant $A$, since the left-hand side is a function of $t$ and the expression in the middle is defined on $N$. 
Moreover, if we take $\tilde{v}_N=v_N+A$, we have $v=\varphi \tilde v_N$, so this can be considered as a solution with $\alpha=0$. Hence, we conclude that $\partial_{x_i}r(X)=0$ for all $i=1,\dots, n-1$ and all eigenvectors $X$ (name it Case 1), or $v$ splits as $v=\varphi \tilde v_N$ (name it Case 2). We analyze the two possibilities separately.

{\bf Case 1.} Assume that $\partial_{x_i}r(X)=0$ for all $i=1,\dots, n-1$ and all eigenvectors $X$ on an open set. Then we have
\[
	0=-m \partial_{x_i}h+(r-(m+n-2)(\nu^2 -4\lambda\widehat{\lambda}) )\partial_{x_i}v_N,
\]
which can be integrated and solved for $h$ to find
\begin{equation}\label{eq:hrconst-WEconf}
	h=m^{-1}(r-(m+n-2)(\nu^2-4\lambda \widehat\lambda))v_n-\xi
\end{equation}
for some constant $\xi$. Substituting this value of $h$ into \eqref{eq:conf-QE-fiber-hr}, and using $\varphi''=-2\lambda \varphi$ and \eqref{eq:const_beta_Eins_QE} yields
\begin{equation}\label{eq:rconsthint-WEconf}
	0=\alpha(r-(n-2)(\nu^2-4\lambda\widehat\lambda )-2m\lambda \varphi^2)+m\varphi(\xi+\kappa \varphi-\alpha'\varphi').
\end{equation}
Notice that, since $\alpha=0$ leads to Case 2, we can further assume $\alpha\neq 0$ in this instance.
Dividing by $\varphi$ and differentiating with respect to $t$, we simplify using that $\varphi''=-2\lambda \varphi$ and $\alpha''=-2\lambda \alpha+\kappa$ to obtain
\[
	0=\frac{(r-(n-2)(\nu^2-4\lambda \widehat\lambda))(\varphi \alpha'-\alpha \varphi')}{\varphi^2},
\]
so, locally, either $\varphi \alpha'-\alpha \varphi'$  or $r=(n-2)(\nu^2-4\lambda \widehat\lambda)$. The first case implies $\alpha= C\varphi$ for some $ C\in \mathbb{R}$, so it can be reformulated as $\alpha=0$ again, thus fitting into Case 2 below.

Assume  $r=(n-2)(\nu^2-4\lambda \widehat\lambda)$. Since $r=\rho^N(X,X)$ for an arbitrary Ricci eigenvector, it follows that
\begin{equation}\label{eq:rhoN-Einstein-2WE}
	\rho^N=(n-2)(\nu^2-4\lambda \widehat\lambda)g^N=(n-2)\left(2\lambda+\frac{(\varphi')^2}{\varphi^2}\right)g
\end{equation}
and the fiber $(N,g^N)$ is Einstein. Additionally, by  \eqref{eq:ricci-warped}, we have
\[
	\rho(X,Y)=\rho^N(X,Y)-\left(\frac{\varphi''}{\varphi}+(n-2)\frac{(\varphi')^2}{\varphi^2}\right)g(X,Y)=2(n-1)\lambda g(X,Y),
\]
so the Ricci tensor satisfies $\rho=2(n-1)\lambda g$, i.e., $(M,g)$ is Einstein on an open set.  Moreover,  \eqref{eq:rconsthint-WEconf} becomes
\begin{equation}\label{eq:coef-gamma-2WE-conf}
	m\varphi(\xi+(\kappa-2\lambda \alpha)\varphi-\alpha'\varphi')=0,
\end{equation}
from where $\xi=\alpha' \varphi'-(\kappa-2\lambda \alpha)\varphi$, which is indeed a constant due to the equations $\varphi''=-2\lambda \varphi$ and $\alpha''=-2\lambda \alpha+\kappa$. Moreover, note that this shows that $\xi$ does not depend on the choice of the eigenvalue $X$. Equation \eqref{eq:hrconst-WEconf} now yields that $h(X)=-(\xi+(\nu^2-4\lambda \widehat\lambda) v_N)$ for every eigenvector $X$, so 
 $v_N$ satisfies the generalized Obata equation
\begin{equation}\label{eq:obata-eq-fiber-confWE}
	\Hes^N_{v_N}=-(\xi+(\nu^2-4\lambda \widehat\lambda) v_N) g^N.
\end{equation}

Furthermore, the form of the conformal factor $u$ given by Lemma~\ref{le:weighted-Einstein-conformal-Obata} around its regular points yields $\widehat{\rho}=2(n-1)\widehat{\lambda}\widehat{g}$ (see Remark~\ref{le:conf-trans-rho-rhofm} for details). This corresponds to Theorem~\ref{th:conf-class-2WE-local}~(1) on a suitable open set.

{\bf Case 2.} Consider now solutions with $\alpha=0$, i.e., those whose density is $v=\varphi v_N$. From the fact that $\alpha''=-2\lambda \alpha+\kappa$, it follows that $\kappa=0$, so $(M,g,f,m)$ is quasi-Einstein. Moreover, from \eqref{eq:BE-Ricci-base} we obtain that $\rho_f^m=2(m+n-1)\lambda g$. With this, equation~\eqref{eq:conf-QE-fiber} takes the form
\begin{equation}\label{eq:quasi-Einstein-fiber-2WE}
\begin{array}{rcl}
	\left(\rho^N-m{v_N}^{-1}\Hes_{v_N}\right)(X,Y)&=&(m+n-2)(2\lambda \varphi^2+(\varphi')^2)g^N(X,Y) \\
	\noalign{\medskip}
	&=&(m+n-2)(\nu^2 -4\lambda\widehat{\lambda})g^N(X,Y)
\end{array}
\end{equation}
where we have also used equation~\eqref{eq:const_beta_Eins_QE}. 
The left-hand side in \eqref{eq:quasi-Einstein-fiber-2WE} is the Bakry-\'Emery Ricci tensor $\rho^m_{f_N}$ on $(N,g^N,f_N=-m\log v_N,m)$, so the fiber is also quasi-Einstein. 
Now  the form of $u$ given by Lemma~\ref{le:weighted-Einstein-conformal-Obata} leads to  $\widehat{\rho}_f^m=2(m+n-1)\widehat{\lambda}\widehat{g}$ (see Remark~\ref{le:conf-trans-rho-rhofm} for details), so Theorem~\ref{th:conf-class-2WE-local}~(2) follows on a suitable open set.

Finally, since weighted Einstein manifolds and their densities are real analytic in harmonic coordinates (see \cite[Theorem 2.5]{Brozos-Mojon-JMPA}), and $(M,g)$ and $(M,\widehat g)$ are Einstein (respectively, quasi-Einstein) on an open set, they are Einstein (respectively, quasi-Einstein) everywhere. \qed


\subsection{Remarks and examples}

In this section, we will provide some additional information on the weighted Einstein SMMSs that appear in Theorem~\ref{th:conf-class-2WE-local}. We will also construct some examples of such SMMSs and explicit conformal factors that transform them into conformally equivalent weighted Einstein manifolds.

\begin{remark}\label{le:conf-trans-rho-rhofm}
The conformal factor in Theorem~\ref{th:conf-class-2WE-local} transforms Einstein manifolds into Einstein manifolds (item (1)) and quasi-Einstein manifolds into quasi-Einstein manifolds (item (2)). The fist case corresponds to conformal transformations that have been already described in the literature (see \cite{Kuhnel}), so only the second case needs to be checked. Nevertheless, in order to be self-contained and justify the precise description in Theorem~\ref{th:conf-class-2WE-local} we include the details of both transformations as follows.

The underlying manifolds of SMMSs described in  Theorem~\ref{th:conf-class-2WE-local}~(1) are Einstein with $\rho=2(n-1)\lambda g$. In this case, the conformally transformed manifold $(M,\widehat{g})$ where $\widehat{g}=u^{-2}g$ is also Einstein with $\widehat{\rho}=2(n-1)\widehat{\lambda}\widehat{g}$. Indeed, by the transformation formula for the Ricci tensor (see, for example, \cite{Kuhnel-Rademacher}),
	\[
	\widehat{\rho}=\rho+u^{-2}((n-2)u\Hes_u+(u\Delta u) g-(n-1)\|\nabla u\|^2 g),
	\]
the conformal factor is a concircular function ($\Hes_u=(\nu-2\lambda u)g$),
	and since  $(u')^2=-2\lambda u^2+2\nu u-2\widehat{\lambda}$ by Lemma~\ref{le:weighted-Einstein-conformal-Obata}, the value of $\widehat{\rho}$ follows.
	
 For SMMSs in Theorem~\ref{th:conf-class-2WE-local}~(2), $(M,g,f,m)$ is quasi-Einstein with $\rho^m_f=2(m+n-1)\lambda g$. We check that the conformal transformation preserves the quasi-Einstein character as follows. The Bakry-Émery Ricci tensor transforms under a weighted conformal change $\widehat{g}=u^{-2}g$, $\widehat{v}=u^{-1}v$ as (see \cite{Case-QE-SMMS})
	\[
	\widehat{\rho}_f^m=\rho_f^m+(m+n-2)u^{-1}\Hes_u+(u^{-1}(\Delta u-g(\nf,\nabla u))-(m+n-1)u^{-2}\|\nabla u\|^2)g.
	\] 
	Moreover, by Lemmas~\ref{le:weighted-Einstein-conformal-Obata} and \ref{le:conf-2WE-IorN-QE}, and the proof of Theorem~\ref{th:conf-class-2WE-local}, we have  $\varphi=\pm u'$, $u''=\nu-2\lambda u$ and $f=-m\log(\varphi v_N)$. Hence,  $g(\nf,\nabla u)=-m u''=-m(\nu-2\lambda u)$ and substituting these values into the expression above yields
	$\widehat{\rho}_f^m=2(m+n-1)\widehat{\lambda}\widehat{g}$,
	so the transformed manifold is quasi-Einstein. Hence, weighted conformal transformations of manifolds in Theorem~\ref{th:conf-class-2WE-local}~(1) stay in that family, and the same is true for Theorem~\ref{th:conf-class-2WE-local}~(2).
\end{remark}

\begin{remark}\label{re:split-v-eins-2WE}
	For non-trivial manifolds in Theorem~\ref{th:conf-class-2WE-local}~(1), a warped product decomposition similar to that given by Lemma~\ref{le:weighted-Einstein-conformal-Obata} arises with respect to the density $v$. Indeed, since the underlying manifold $(M,g)$ is Einstein  with $\rho=2(n-1)\lambda g$, the weighted Einstein equation \eqref{eq:BE-Ricci-warped2} reads
	\begin{equation}\label{eq:gen-obata-conf-v}
 				\Hes_v=(2\lambda v -\kappa)g,
	\end{equation}
so the arguments from Lemma~\ref{le:weighted-Einstein-conformal-Obata} can be mimicked to split $M$ around regular points of $v$, with the fibers being level hypersurfaces of $v$. This approach has the advantage of $v$ depending only on a coordinate of the base, but in general this will no longer apply to $u$.

Moreover, for the same manifolds in Theorem~\ref{th:conf-class-2WE-local}~(1) and the warped product splitting $I\times_\varphi N$ given in the theorem (i.e. with the fibers being level hypersurfaces of $u$), \eqref{eq:obata-eq-fiber-confWE} is again an Obata-type equation holding on the fiber. Hence, it induces an additional splitting of the form $N=I_2\times_{\varphi_2}N_2$ around regular points of $v_N$, where the fibers are level hypersurfaces of $v_N$. This can be used to extract further information about the geometry of $N$, using known properties of the solutions of the generalized Obata equation (see, for example, \cite{Wu-Ye-Obata}).
\end{remark}

\begin{remark}\label{re:conf-2WE-local-classification} 
Let  $(M^n,g,f,m,\mu)$ be a SMMS in the hypotheses of Theorem~\ref{th:conf-class-2WE-local}~(1), where $M=I\times_\varphi N$ is Einstein with $\rho=2(n-1)\lambda g$ and $N$ is also Einstein with $\rho^N=\beta g^N$. Let $t$ be a coordinate parameterizing $I$ by arc length and assume $v=v(t)$. These SMMSs can be described very explicitly. Indeed, by solving the various ODEs that appear in the proof of the theorem it is straightforward to see that $v$, $\varphi$, $\beta$ and $\mu$ take the following forms (the value of $\mu$ being irrelevant if $m=1$): 

\begin{spacing}{1.5}
\begin{tabular}{|c|c|c|}
	\hline
$\lambda>0$ & 			$\varphi (t)=a \cos(t\sqrt{2\lambda})+b \sin(t\sqrt{2\lambda})$& $\beta=2 (a^2+b^2)(n-2)\lambda$\\
\cline{2-3}
& $v(t)=\frac{\kappa}{2\lambda}-bc\cos(t\sqrt{2\lambda})+ac\sin(t\sqrt{2\lambda})$ & $\mu=2 (a^2+b^2)c^2\lambda-\frac{\kappa^2}{2\lambda}$\\
\cline{2-3}
&  $u(t)=\frac{\nu}{2\lambda}  \mp \frac{b}{\sqrt{2\lambda}}\cos(t\sqrt{2\lambda})  \pm\frac{a}{\sqrt{2\lambda}}\sin(t\sqrt{2\lambda})$ & $\widehat{\lambda}=\frac{\nu^2}{4\lambda}- \frac{a^2+b^2}2$\\
\hline
 $\lambda=0$ & $\varphi (t)= a \kappa t+b$ & $\beta=a^2\kappa^2(n-2)$\\
\cline{2-3}
 & $v (t)=\frac{\kappa}{2}t^2+ct+d$ 	where $b=ac$ if $\kappa \neq 0$ & $\mu=c^2-2d\kappa$\\
\cline{2-3}
&$ u(t)=\frac{\nu  t^2}{2}\pm b t+l \qquad (\nu=\pm a\kappa)$& $\widehat{\lambda}=-\frac12b^2+l\nu$ \\
\hline
$\lambda<0$ & $	\varphi (t)=ae^{t\sqrt{-2\lambda}}+be^{-t\sqrt{-2\lambda}}$ & $\beta=8 a b(n-2)\lambda$\\
\cline{2-3}
& $v(t)=\frac{\kappa}{2\lambda}+ace^{t\sqrt{-2\lambda}}-bce^{-t\sqrt{-2\lambda}}$ & $\mu=-8abc^2\lambda-\frac{\kappa^2}{2\lambda}$\\
\cline{2-3}
 & $u(t)=\frac{\nu }{2
	\lambda }\pm  \frac{a }{\sqrt{-2\lambda }}e^{t\sqrt{-2\lambda}}\mp\frac{b }{\sqrt{-2\lambda }}e^{-t\sqrt{-2\lambda}}  $ & $\widehat{\lambda}=\frac{\nu^2}{4\lambda}- 2ab$\\
\hline
\end{tabular} 
\end{spacing}

\medskip

The constants $a,b,c,d$ and the scale $\kappa$ are such that $\varphi$, $v$ and $u$ are positive on~$I$. 

Notice that the density and the conformal factor present qualitatively the same behavior, up to some constant translation and rescaling. This allows for conformal changes between trivial and non-trivial SMMSs by suitably adjusting the constants involved (cf. Remark~\ref{rmk:analysis-change-constants-global}). 

Moreover, observe that in the three cases the value of $\widehat \lambda$ after the conformal change can be any real number. This contrasts with the behavior of global conformal changes (see Theorem~\ref{th:conf-class-2WE-complete} and Corollary~\ref{cor:2WEconf-compact}).
\end{remark}

\begin{remark}\label{re:2WE-conf-harmonic}
Note that the Riemann curvature tensor of an Einstein manifold decomposes as $R=\frac{\tau}{2n(n-1)}g\KN g+W$, where $\KN$ denotes the Kulkarni-Nomizu product. For manifolds in Theorem~\ref{th:conf-class-2WE-local}~(1), $\rho=2(n-1)\lambda g$, so $\tau=2n(n-1)\lambda$ and the curvature tensor satisfies $R=\lambda g\KN g+W$. Additionally, the weighted Weyl tensor is defined as $W_f^m=R-P_f^m\KN g$ and, since $P_f^m=\lambda g$, we have $W=W_f^m=R-\lambda g\KN g$.
	
The warped product splitting $I\times_\varphi N$ with $N$ Einstein guarantees that the Weyl tensor is harmonic, i.e. $\delta W=0$, where $\delta$ denotes divergence. Also, if we consider the alternative warped product splitting discussed in Remark~\ref{re:split-v-eins-2WE} around regular points of  the density $v$, it is straightforward to see that $W_f^m(\nf,-,-,-)=0$. Hence, it follows that the weighted divergence $\delta_fW_f^m=\delta W_f^m-\iota_{\nf} W_f^m$ vanishes too. Non-trivial weighted Einstein SMMSs with weighted harmonic Weyl tensor were classified in \cite{Brozos-Mojon-JMPA} and we point out that non-trivial SMMSs in Theorem~\ref{th:conf-class-2WE-local}~(1) are exactly those that arise in that context in the case where the underlying manifold $(M,g)$ is Einstein (see \cite[Theorem 4.1]{Brozos-Mojon-JMPA}). In particular, weighted space forms are included in this family, but other kinds of SMMSs which are not weighted  locally conformally flat (i.e., SMMSs with $W_f^m\neq 0$) also appear. 
	
\end{remark}

Although many relatively simple weighted Einstein structures are realized on Einstein manifolds (such as the weighted space forms), more involved structures are admissible as well, like those that fall into Theorem~\ref{th:conf-class-2WE-local}~(2), whose underlying manifold is quasi-Einstein. The following is an example constructed from a quasi-Einstein manifold with $\rho^m_f=0$ that was built in \cite[Example 2]{silvafilho-QE}.

\begin{example}\label{ex:non-Einstein-conf-2WE}\rm
	For  any value of $m\neq 1$, consider the direct product $(N,g_0)=(\tilde{N_2}^{n-2}\times \mathbb{R},g^{\tilde{N}}\oplus  ds^2)$, where $\tilde{N}$ is an Einstein manifold with positive Einstein constant $\xi$,  i.e. $\rho^{\widehat N}=\xi g^{\widehat N}$. 
	Define the function $\sigma:C^\infty(N)\rightarrow \mathbb{R}$ by  $\sigma=\frac{m\sqrt{\xi}}{\sqrt{m+n-3}}\pi_\mathbb{R}$, where $\pi_\mathbb{R}$ denotes the projection on $\mathbb{R}$, and consider the conformal metric $g^N=e^{-2\frac{\sigma}{m}}g_0$. Then
	$(N,g^N,\sigma,m)$ is a quasi-Einstein manifold  with $(\rho^m_\sigma)^N=0$ (see \cite{silvafilho-QE}). 
	
	Next, consider the direct product $(M^n,g)=(I\times N, dt^2\oplus g^N)$, where $I\subset \mathbb{R}$ is an open interval, and set $f=\sigma$. 
	Taking $\mu=\frac{\xi}{m-1}$ makes $(M,g,f,m,\mu)$ a weighted Einstein manifold with $\lambda=0$ (and it is also quasi-Einstein with $\rho_f^m=0$). Moreover, let $u=at+b$ for constants $a\neq 0,b$ such that $u>0$ in $I$. Take the weighted conformal change given by $\widehat{g}=u^{-2}g$ and $\widehat{f}=f+m\log\left(u\right)$. Then, $(M,\widehat{g},\widehat{f},m,\mu)$ is weighted Einstein with $\widehat{\lambda}=-\frac{a^2}{2}<0$, and also quasi-Einstein with $\widehat{\rho}_f^m=-(m+n-1)a^2\widehat{g}$. Note that neither $(M,g)$ nor $(M,\widehat{g})$ are Einstein. Indeed, $\tau=\frac{m(n-2)}{m+n-3}\xi e^{2\frac{f}m}$ and $\widehat{\tau}=u^2\tau-n(n-1)a^2$ are non-constant.
\end{example}

\section{The complete case}\label{sect:2WEconf-complete}

A key point in our analysis of SMMSs with two weighted Einstein structures in the same weighted conformal class is that the conformal factor satisfies the generalized Obata equation \eqref{eq:generalized-Obata-equation-conformal}, but the results stated in the previous section have mostly been centered around the local features of the geometry of the SMMSs of interest. 
Although global solutions to the generalized Obata equation have a long history, dating back to \cite{Brinkmann}, we are going to use a more general result given in \cite{Wu-Ye-Obata}. We summarize it as follows.  


\begin{theorem}\cite[Theorem 4.6]{Wu-Ye-Obata} \label{th:generalized-Obata-equation}
	Let $(M^n,g)$ be a complete Riemannian manifold admitting a non-constant smooth solution $u$ of the generalized Obata equation \eqref{eq:generalized-Obata-equation-conformal} for a smooth function $\gamma$. Then
	\begin{enumerate}
		\item If $u$ has critical points (at most, it can have two), then $(M,g)$ is isometric to a complete extension of the product on $(0,T)\times \mathbb{S}^{n-1}$ with warped metric
		\begin{equation}\label{eq:mfmu}
			g=dt^2+\gamma(\eta)^{-2}(\omega'(t))^2g_{\mathbb{S}^{n-1}},
		\end{equation}
		where $\gamma\in C^\infty(I)$, $I=(a,b)$, $[a,b)$, $(a,b]$ or $[a,b]$ (with $a$, $b$ possibly infinite) and  $\eta\in I$ such that $\gamma(\eta)\neq 0$; and $\omega$ is the unique maximally extended  smooth solution of the initial value problem
		\[
		\omega''+\gamma(\omega)=0, \quad \omega(0)=\eta, \quad \omega'(0)=0.
		\]
		Moreover, $T$ is the (possibly infinite) supremum of $t$ such that $\omega$ is defined on $[0,t]$ and $\omega'\neq 0$ in $(0,t]$; and $u$ is identified with $\omega$ on $M$.
		\item If $u$ does not have critical points, $(M,g)$ is isometric to a warped product $\mathbb{R}\times_\varphi N$, where $N$ is complete and $u$ is defined on the base $\mathbb{R}$.
	\end{enumerate} 
\end{theorem}

\noindent In the proof of Theorem~\ref{th:conf-class-2WE-complete} we make use of this result, taking into account that the function $\gamma$ takes the form $\gamma(u)=2\lambda u-\nu$ given in Lemma~\ref{le:weighted-Einstein-conformal-Obata}.

\bigskip

\noindent {\it Proof of Theorem~\ref{th:conf-class-2WE-complete}}:
Let $(M^n,g,f,m,\mu)$ be a complete SMMS with $P_f^m=\lambda g$, such that there exists a SMMS $(M^n,\widehat{g},\widehat{f},m,\mu)$, with $\widehat{P}_f^m=\widehat{\lambda} \widehat{g}$, related by a non-constant conformal factor $u$, i.e. $\widehat{g}=u^{-2}g$ and $\widehat{v}=u^{-1}v$.


By Lemma~\ref{le:weighted-Einstein-conformal-Obata}  we know that the (globally defined) conformal factor $u$ satisfies the generalized Obata equation $\Hes_u=(\nu-2\lambda u)g$, which allows us to apply Theorem~\ref{th:generalized-Obata-equation}. There are two possibilities that we analyze separately.   First, consider the case where $u$ has critical points. Then, $(M,g)$ is isometric to 
\[
g=dt^2+\varphi(t)^2 g_{\mathbb{S}^{n-1}}, \quad t\in (0,T),
\] 
with $\varphi(t)=\frac{u'(t)}{(2\lambda u(0)-\nu)}$, where 
\[
u''+2\lambda u-\nu=0, \quad u(0)=\xi>0, \quad u'(0)=0.
\]
If $\lambda>0$, it follows that $u(t)=\frac{1}{2\lambda}\left(\nu+(2 \xi\lambda-\nu)\cos(t\sqrt{2\lambda})\right)$, with $t\in (0,\frac{\pi}{\sqrt{2\lambda}})$. The warping function is $\varphi(t)=\frac{1}{\sqrt{2\lambda }}\sin \left(t\sqrt{2\lambda}\right)$. Hence, $(M,g)$ is isometric to a sphere of constant sectional curvature $2\lambda$, which only admits solutions in the form of the $m$-weighted $n$-spheres as described in Example~\ref{ex:sphere}. 
For $\lambda=0$ and $\lambda<0$ analogous processes lead to the $m$-weighted $n$-Euclidean spaces of Example~\ref{ex:euclidean-space} and to the $m$-weighted $n$-hyperbolic spaces of Example~\ref{ex:hyperbolic-space}. Locally, these spaces look like the SMMSs in Theorem~\ref{th:conf-class-2WE-local}~(1). Moreover, by the argument in Remark~\ref{re:2WE-conf-harmonic}, they have weighted harmonic Weyl tensor, so the classification in \cite{Brozos-Mojon-JMPA} also implies this conclusion in the non-trivial case. These are the SMMS in Theorem~\ref{th:conf-class-2WE-complete}~(1).

Now, we consider the other possibility in Theorem~\ref{th:generalized-Obata-equation}, this is,  $u$ has no critical points, which guarantees that $(M,g)$ is isometric to a warped product of the form $\mathbb{R}\times_\varphi N$ where $\nabla u$ is tangent to $\mathbb{R}$. 
Then we use the fact that the functions $v$, $\varphi$ and $u$ are globally defined. The warping function $\varphi$ and the conformal factor $u$ satisfy the ODEs  $u''+2\lambda u-\nu=0$ and $\varphi=\pm u'$ given in  Lemma~\ref{le:weighted-Einstein-conformal-Obata}, and the only global solutions such that $u$ has no critical points and  both $u$ and $\varphi$ remain positive on $\mathbb{R}$ are  (after an inversion of the sign of $t$ if necessary) of the form $\varphi=Ae^{t\sqrt{-2\lambda}}$ and $u(t)=\frac{\nu}{2\lambda}+\frac{A}{\sqrt{-2\lambda}}e^{t\sqrt{-2\lambda}}$ for $\lambda<0$ and some $A>0$, $\nu\leq 0$. Now, the expression of $\varphi$ yields, from \eqref{eq:const_beta_Eins_QE}, $\nu^2-4\lambda \widehat \lambda=2\lambda \varphi^2+(\varphi')^2=0$. Thus, $\widehat{\lambda}=\frac{\nu^2}{4\lambda}\leq 0$, so $u(t)=\sqrt{\frac{ \widehat{\lambda}}{\lambda}}+\frac{A}{\sqrt{-2\lambda}}e^{t\sqrt{-2\lambda}}$.

By Lemma~\ref{le:conf-2WE-IorN-QE}, the density $v$ takes the form $v=\varphi v_N+\alpha$.  
We distinguish between the cases where $\alpha$ is constant and non-constant, and we analyze each one separately. Notice that, if $\alpha$ is a constant, the form of $v$ reduces to $v=\varphi v_N$ and we can assume $\alpha=0$.

{\bf Case 1.} Assume $\alpha$ is non-constant. First, observe that the SMMS falls into Theorem~\ref{th:conf-class-2WE-local}~(1), so $g$ and $\tilde g$ are Einstein.  From \eqref{eq:rhoN-Einstein-2WE}, since $2\lambda \varphi^2+(\varphi')^2=0$, we obtain $\rho^N=0$, i.e. the fiber $N$ is Ricci-flat.

From $\alpha''=-2\lambda \alpha+\kappa$, it follows that $\alpha(t)=\frac\kappa{2\lambda}+Be^{t\sqrt{-2\lambda}}+Ce^{-t\sqrt{-2\lambda}}$ for some $B,C\geq 0$ such that $v>0$. Moreover, for this form of $\alpha$, equation~\eqref{eq:coef-gamma-2WE-conf} reduces to  $\xi=4AC\lambda$, so the generalized Obata equation \eqref{eq:obata-eq-fiber-confWE} for the fiber takes the form $\Hes_{v_N}=-4AC\lambda g^N$. Hence, either $v_N$ is constant and $C=0$; or $C>0$.
\begin{itemize}
	\item If $v_N$ is constant and $C=0$, redefining the constant $B$ if necessary, we can take $v=\alpha(t)=\frac\kappa{2\lambda}+Be^{t\sqrt{-2\lambda}}$. In order for $v$ to stay positive in $\mathbb{R}$, it follows that $\kappa \leq 0$.
	Moreover, a straightforward calculation from the weighted Einstein equation \eqref{eq:weighted-Einstein-equation} yields that $m=1$ or  $\mu=-\frac{\kappa ^2}{2 \lambda }\geq0$. This corresponds to Theorem~\ref{th:conf-class-2WE-complete}~(2)(a).
	\item  If $C>0$, from equation $\Hes_{v_N}=-4AC\lambda g^N$, it follows that $N$ is isometric to $\mathbb{R}^{n-1}$ and $v_N(r)=-2AC\lambda r^2+D $, where $r$ is the radial coordinate around some point of $\mathbb{R}^{n-1}$ (see  \cite[Theorem 6.3]{Wu-Ye-Obata}). The manifold $\mathbb{R}\times_\varphi \mathbb{R}^{n-1}$ is isometric to the hyperbolic space. In order to see that this case corresponds to Example~\ref{ex:hyperbolic-space}, we keep analyzing the density function. Since $A,C> 0$ and $\lambda<0$, it follows that $v$ remains positive for all values of $r$ if and only if $D\geq0$. 
	Redefining the constant $B$, we can write $v=\frac\kappa{2\lambda}+Be^{t\sqrt{-2\lambda}}+Ce^{-t\sqrt{-2\lambda}}- 2A^2C\lambda r^2 e^{t\sqrt{-2\lambda}}$, 
	so $v$ necessarily has a critical point. Now, applying the splitting given in Remark~\ref{re:split-v-eins-2WE} and knowing that $v$ has critical points, by Theorem~\ref{th:generalized-Obata-equation}, we conclude that this case corresponds to Example~\ref{ex:hyperbolic-space}. Additionally, with this expression for $v$, one checks that these SMMSs satisfy $W=W_f^m=0$, so the conclusion also follows from the classification in \cite{Brozos-Mojon-JMPA}.
\end{itemize} 

{\bf Case 2.} Assume $\alpha=0$. Notice that, in this case, the SMMS falls into Theorem~\ref{th:conf-class-2WE-local}~(2).  Then, equation \eqref{eq:quasi-Einstein-fiber-2WE} guarantees that $(N,g^N,f_N=-m\log v_N,m)$ is quasi-Einstein with $(\rho_{f_N}^m)^N=0$. This yields the quasi-Einstein manifolds in Theorem~~\ref{th:conf-class-2WE-complete}~(2)(b). 
\qed

\bigskip

The underlying manifold of possible solutions in Theorem~\ref{th:conf-class-2WE-complete}~(1) and (2)(a) are Einstein and therefore correspond to Einstein metrics with another Einstein metric in its conformal class as described in \cite{Kuhnel}. Hence, the genuinely new metric structures in Theorem~\ref{th:conf-class-2WE-complete} are those given in Theorem~\ref{th:conf-class-2WE-complete}~(2)(b). In order to build an example of this kind, it suffices to take any non-trivial complete quasi-Einstein manifold with $\rho^m_f=0$ as a fiber. The following one illustrates this fact, using a complete fiber related to the construction  of the {\it generalized Schwarzschild metric} (see \cite[Example 9.118a]{Besse} and \cite{He-warped-Einstein} for more details).

\begin{example}
For $m>1$, consider the two-dimensional manifold $N=\mathbb{R}^2$ endowed with the warped product metric $g^N=dx^2+(\omega'(x))^2d\theta^2$, where $\omega$ is the unique solution on $[0,\infty)$ of the problem
\[
(\omega')^2=1-\omega^{1-m}, \quad \omega(0)=1, \quad \omega \geq 0.
\]
For example, if $m=3$, then $\omega(x)=\sqrt{1+x^2}$. This metric extends smoothly across $x=0$, and the resulting manifold $(N,g^N, -m \log(\omega), m)$ is complete and steady quasi-Einstein, but $\rho^N$ is not, in general, a constant multiple of the metric (indeed, for $m=3$, $\rho^N=\frac{3}{(1+ x^2)^2}g^N$). Thus, for $(\mathbb{R}^3,g=dt^2+\varphi^2g^N)$, $\lambda<0$, $\varphi(t)=e^{\sqrt{-2\lambda}t}$ and $f(t,x,\theta)=-m\left(\sqrt{-2\lambda}\, t+\log(\omega(x))\right)$, the SMMS $(\mathbb{R}^3,g, f, m, \mu)$ is quasi-Einstein with $\rho_f^m=2(m+2)\lambda g$ and thus, weighted Einstein for an appropriate $\mu$ (for $m=3$, the value is $\mu=1$). Moreover, for $\widehat{\lambda}<0$ and the conformal factor $u(t)=\sqrt{\frac{ \widehat{\lambda}}{\lambda}}+\frac{1}{\sqrt{-2\lambda}}e^{\sqrt{-2\lambda}t}$, the conformally transformed SMMS $(\mathbb{R}^3,u^{-2}g, f+m\log(u), m, \mu)$ is also quasi-Einstein with $\widehat{\rho}_f^m=2(m+2)\widehat{\lambda}\widehat{g}$, hence weighted Einstein with $\widehat{P}_f^m=\widehat{\lambda} \widehat{g}$.
\end{example}

%

It was shown in \cite{Kuhnel} that the existence of a solution to the generalized Obata equation forces the manifold to be conformally equivalent to a sphere. In the context of this paper, this relates to Corollary~\ref{cor:2WEconf-compact}, that we obtain as a consequence of Theorem~\ref{th:conf-class-2WE-complete}. 

\medskip

\noindent {\it Proof of Corollary~\ref{cor:2WEconf-compact}}:
Out of the admissible geometries for complete SMMSs given by Theorem~\ref{th:conf-class-2WE-complete}, the only compact ones are the weighted spheres given by Example~\ref{ex:sphere}. Hence, $\lambda>0$ and $(M^n,g,f,m,\mu)$ is globally isometric to such a sphere, so $v(t)=A+B\cos(\sqrt{2\lambda} t)$ for some constants $A\in \mathbb{R}^+$, $B\in \mathbb{R}$ such that $A>|B|$. Now, we can take the conformal factor $u=v$, so that $\widehat f=f+m\log u=0$.

Then, it is straightforward to prove that the conformally transformed SMMS $(M,u^{-2}g,0,m,\mu)$ has constant sectional curvature $\widehat{\lambda}=(A^2-B^2)\lambda>0$, so it is isometric to the sphere $(\mathbb{S}^n(2\widehat \lambda),g_{\mathbb{S}}^{2\widehat{\lambda}},0,m,\mu)$, which is a trivial weighted Einstein manifold with $\widehat{P}_f^m=\widehat{\lambda}g_{\mathbb{S}}^{2\widehat{\lambda}}$.
\qed

\begin{remark}
Note that, according to Example~\ref{ex:sphere}, the curvature parameter $\mu$ for the weighted spheres is given by $\mu=2\lambda(B^2-A^2)$. Thus, the condition $A>|B|$ guarantees that the only weighted spheres with $\mu=0$ are those with $m=1$, where $\mu$ does not play a role. Thus, this recovers the result in \cite{Case-JDG}.
\end{remark}

\begin{example}\rm 	
As a consequence of the fact that the density $v$ satisfies \eqref{eq:gen-obata-conf-v}, the only weighted Einstein structures on standard spheres are the $m$-weighted $n$-spheres portrayed in Example~\ref{ex:sphere}.	However, since the symmetry of the sphere allows for the poles to be any two antipodal points, the conformal factor does not necessarily vary in the same direction as the density, and it can be used to rotate it and modify its radius while maintaining its weighted Einstein character. For example,  let $\lambda>0$ and consider the sphere $(\mathbb{S}^n,g_{\mathbb{S}}^{2\lambda})$, whose metric can be written as
\[
g_{\mathbb{S}}^{2\lambda}=dt^2+(2\lambda)^{-1}\sin^2( \sqrt{2\lambda} t)( d\theta^2+\sin^2( \theta) g_{{\mathbb{S}}^{n-2}}), \quad t\in \left(0,\tfrac{\pi}{\sqrt{2\lambda}}\right), \,\,\theta\in \left(0,\pi\right).
\]
The warping function $\varphi(t)=\frac{\sin(t\sqrt{2\lambda})}{\sqrt{2\lambda}}$ is induced by the conformal transformation given by $u(t)=\frac{\nu}{2\lambda}-\frac{\cos(t\sqrt{2\lambda})}{2\lambda}$, with $\nu>1$, since $\varphi=u'$. Now, instead of the density given in Example~\ref{ex:sphere}, consider $v(t,\theta)=A\varphi(t) \cos(\theta)+B\cos(t\sqrt{2\lambda})+\frac{\kappa}{2\lambda}$ (where, $A,B,\kappa$ are such that $v$ is always positive). The corresponding SMMS $(\mathbb{S}^n,g^{2\lambda},f,m,\mu)$ is weighted Einstein with scale $\kappa$ and $P_f^m=\lambda g$ for $\mu=A^2+2\lambda B^2-\frac{\kappa^2}{2\lambda}$ or $m=1$. Notice that $v$ is of the form given in Lemma~\ref{le:conf-2WE-IorN-QE} with  $\alpha(t)=\frac{\kappa}{2\lambda}+B\cos(t\sqrt{2\lambda})$, so that $\xi=\alpha' \varphi'-(\kappa-2\lambda \alpha)\varphi=0$, where $\xi$ is the constant in the generalized Obata equation~\eqref{eq:obata-eq-fiber-confWE}.

Moreover, the transformed manifold $(\mathbb{S}^n,\widehat g,\widehat f,m,\mu)$, is weighted Einstein with $\widehat P_f^m=\widehat \lambda \widehat g$ and $\widehat \lambda =\frac{\nu^2-1}{4\lambda}>0$. In this case, $v_N(\theta)=A\cos(\theta)$ is indeed a solution of the generalized Obata equation \eqref{eq:obata-eq-fiber-confWE} on $\mathbb{S}^{n-1}$ for $\xi=0$ and $\nu^2-4\lambda \widehat{\lambda}=1$, as stated in Theorem~\ref{th:conf-class-2WE-local}~(1), and the metric $\widehat g$ has constant sectional curvature $2\widehat \lambda $, so $(\mathbb{S}^n,\widehat g,\widehat f,m,\mu)$ is a weighted sphere.
\end{example}

\begin{remark}\label{rmk:analysis-change-constants-global}
The proof of Corollary~\ref{cor:2WEconf-compact} relies on the fact that we can take the conformal factor $u$ to be equal to the density function $v$ so that $\widehat f=f+m\log u=0$. This is possible for most of the other Einstein manifolds in Theorem~\ref{th:conf-class-2WE-complete}. However, the sign of the weighted Einstein constant is not necessarily preserved under such transformations when $\lambda\leq 0$, in constrast to the case $\lambda>0$.

\underline{$\lambda=0$:} For the $m$-weighted $n$-Euclidean space as in Example~\ref{ex:euclidean-space} and the conformal change $u(t)=v(t)=A+Bt^2$, $(M^n, \widehat{g}=u^{-2} g_E ,0,m,\mu)$ is isometric to the punctured sphere $(\mathbb{S}^n\setminus~\{N\},g_S^{2\widehat{\lambda}},0,m,\mu)$ with $\widehat{\lambda}=2AB>0$, which is weighted Einstein with $\widehat{P}_f^m=\widehat{\lambda}g_{\mathbb{S}}^{2\widehat{\lambda}}$. Indeed, note that $\widehat{g}=u^{-2}g$ is essentially the change of the metric given by stereographic projection.

\underline{$\lambda<0$:}
For the $m$-weighted $n$-hyperbolic space of constant sectional curvature $2\lambda$ as in Example~\ref{ex:hyperbolic-space}  and the conformal change given by $u(t)=v(t)=A+B\cosh(\sqrt{2\lambda} t)$, $(M^n,\widehat{g}=u^{-2} g_H^{2\lambda},0,m,\mu)$ is weighted Einstein with $\widehat{\lambda}=(A^2-B^2)\lambda$, and it has constant sectional curvature $2\widehat{\lambda}$. In this case, $\widehat{\lambda}$ can be positive, zero or negative depending on the values of $A$ and $B$ (since $A>-B$), so $(M,\widehat{g})$ can be isometric to a punctured sphere, a Euclidean space or a hyperbolic space.

On the other hand, for a warped product $\mathbb{R}\times_\varphi N$ as in Theorem~\ref{th:conf-class-2WE-complete}~(2)(a), where $u(t)=v(t)=\frac{\kappa}{2\lambda}+\frac{A}{\sqrt{-2\lambda}} e^{t\sqrt{-2\lambda}}$, $N$ is a Ricci-flat complete manifold and $\varphi(t)= A e^{t\sqrt{-2\lambda}}$, we have that for $\widehat{\lambda}=\frac{\kappa^2}{4\lambda}\leq 0$, the transformed SMMS $(M,\widehat{g},0,m,\mu)$ satisfies $\widehat{P}_f^m=\widehat{\lambda} \widehat{g}$.

In contrast to the previous items, in the cases in  Theorem~\ref{th:conf-class-2WE-complete}~(2)(b) where $v_N$ is non-constant, the conformal factor is defined on the base of the warped product $I\times_\varphi N$, while the density has a non-constant component in $N$, so we cannot take $u=v$. Nevertheless, in  Theorem~\ref{th:conf-class-2WE-complete}~(2)(b)  we can take $u=\varphi$, in which case the transformed density $\widehat{f}$ will be constant on the base of the product.
\end{remark}


\end{document}